\documentclass[a4paper,12pt, twoside, reqno]{amsart}
\usepackage{amsmath}%
\usepackage{amsfonts}%
\usepackage{amssymb}%
\usepackage{graphicx}
\newtheorem{theorem}{Theorem}
\theoremstyle{plain}

\newtheorem{corollary}{Corollary}

\newtheorem{definition}{Definition}
\newtheorem{example}{Example}

\newtheorem{remark}{Remark}

\newtheorem{ass}{Assumption}[section]
\numberwithin{equation}{section}

\begin{document}
\title{Coupled coincidence point theorems for nonlinear contractions in partially ordered metric spaces}
\author{Vasile Berinde}

\begin{abstract}
We obtain  coupled coincidence and coupled common fixed point theorems for mixed $g$-monotone nonlinear operators $F:X \times X \rightarrow  X$ in partially ordered metric spaces. Our results are generalizations of recent coincidence point theorems due to Lakshmikantham and \' Ciri\' c [Lakshmikantham, V., \' Ciri\' c, L.,  \textit{Coupled fixed point theorems for nonlinear contractions in partially ordered metric spaces}, Nonlinear Anal. \textbf{70} (2009), 4341–-4349], of coupled fixed point theorems established by  Bhaskar and Lakshmikantham [T.G. Bhaskar, V. Lakshmikantham, \textit{Fixed point theorems in partially ordered metric spaces and applications}, Nonlinear Anal. \textbf{65} (2006) 1379-1393] and also include as particular cases several related results in very recent literature.
\end{abstract}
\maketitle

\pagestyle{myheadings} \markboth{Vasile Berinde} {Coupled coincidence point theorems for nonlinear contractions }

\section{Introduction and preliminaries} 

A very recent trend in metrical fixed point theory, initiated by Ran  and  Reurings  \cite{Ran}, and continued by Nieto  and  Lopez  \cite{Nie06}, \cite{Nie07}, Bhaskar  and  Lakshmikantham  \cite{Bha}, Agarwal  et  al.  \cite{Aga}, Lakshmikantham and Ciric \cite{LakC}, Luong and Thuan \cite{Luong} and many other authors, is to consider a partial order on the ambient metric space $(X,d)$ and to transfer a part of the contractive property of the nonlinear operators into its monotonicity properties. This approach turned out to be very productive, see for example \cite{Aga}, \cite{Choud}-\cite{Ran}, and the obtained results found important applications to the existence of solutions for matrix equations or ordinary differential equations and integral equations, see  \cite{Bha}, \cite{Luong}, \cite{Nie06}, \cite{Nie07}, \cite{Ran} and references therein.


In this context, the main novelty brought by Bhaskar  and  Lakshmikantham \cite{Bha} and then continued by Lakshmikantham and Ciric \cite{LakC} and other authors, was to consider nonlinear bivariate mappings $F:X \times X \rightarrow  X$ in direct connection with their so called mixed monotone property, and to study the existence (and uniqueness) of \textit{coupled fixed points} for such mappings.

To fix the context in which we are placing our results, recall the following notions. Let$\left(X,\leq\right)$ be a partially ordered set and  endow the product space $X \times X$ with the following partial order:$$\textnormal{for } \left(x,y\right), \left(u,v\right) \in X \times X,  \left(u,v\right) \leq \left(x,y\right) \Leftrightarrow x\geq u,  y\leq v.$$

We say that a mapping $F:X \times X \rightarrow X$ has the \textit{mixed monotone property} if $F\left(x,y\right)$ is monotone non-decreasing in $ x$ and is monotone non-increasing in $y$, that is, for any $x, y \in X,$   
$$  x_{1}, x_{2} \in X,  x_{1} \leq  x_{2} \Rightarrow  F\left(x_{1},y\right) \leq  F\left(x_{2},y\right) $$
  and, respectively,
$$ y_{1},  y_{2} \in  X, y_{1} \leq y_{2} \Rightarrow F\left(x,y_{1}\right)  \geq F\left(x,y_{2}\right). $$

A pair $ \left(x,y\right) \in X \times X $ is called a \textit{coupled  fixed  point} of the mapping $F$ if                                              $$ F\left(x,y\right) = x,  F\left(y,x\right) = y.$$
The next theorem is the main theoretical result in \cite{Bha}.

\begin{theorem}[Bhaskar and Lakshmikantham \cite{Bha}]\label{th1}
	Let  $\left(X,\leq\right)$ be a partially ordered set and suppose there is a metric $d$  on $X$ such that $\left(X,d\right)$  is a complete metric space. Let $F : X \times X \rightarrow X $ be a continuous mapping having the mixed monotone property on $X$. Assume that there exists a constant $k \in \left[0,1\right)$  with                                                                                   
\begin{equation} \label{Bhas}
	d\left(F\left(x,y\right),F\left(u,v\right)\right) \leq \frac{k}{2}\left[d\left(x,u\right) + d\left(y,v\right)\right],\textnormal{ for each }x \geq u, y \leq  v.
\end{equation}  
If there exist $x_{0}, y_{0} \in X$ such that  
$$                                                                                       
x_{0} \leq F\left(x_{0},y_{0}\right)\textrm{ and }y_{0} \geq F\left(y_{0},x_{0}\right),
$$
 then there exist $x, y \in X$ such that $$x = F\left(x,y\right)\textnormal{ and }y = F\left(y,x\right).$$ 
	\end{theorem}        
	
As shown in \cite{Bha}, the continuity assumption of $F$ in Theorem \ref{th1} can be replaced by the following property imposed on the ambient space $X$:

\begin{ass} \label{1.1}
$X$ has the property that

(i) if a non-decreasing sequence $\{x_n\}_{n=0}^{\infty}\subset X$ converges to $x$, then $x_n\leq x$ for all $n$;

(ii) if a non-increasing sequence $\{x_n\}_{n=0}^{\infty}\subset X$ converges to $x$, then $x_n\geq x$ for all $n$;
\end{ass}
	
 These results were then extended and generalized by several authors in the last five years, see  \cite{LakC}, \cite{Luong} and references therein, to restrict citing only the ones strictly related to our approach in this paper. Amongst these generalizations, we refer especially to the one obtained in \cite{LakC}, which considered instead of \eqref{Bhas} a more general contractive condition
and established corresponding coincidence point theorems.

The following concepts were introduced in \cite{LakC}.

\begin{definition}\label{def1}
 Let $ \left(X,\leq \right)$  be a partially ordered set and  $F:X \times X \rightarrow X, g: X \rightarrow X$. We say that $F$ has the mixed $g$-monotone property if $F$ is monotone $g$-non-decreasing in its first argument and is monotone $g$-non-increasing in its second argument, that is,  for any     $ x, y \in X,$   
$$  x_{1}, x_{2} \in X,  g\left(x_{1}\right) \leq  g\left(x_{2}\right) \Rightarrow  F\left(x_{1},y\right) \leq  F\left(x_{2},y\right), 
$$  
and
$$ y_{1},  y_{2} \in  X, g\left(y_{1}\right) \leq g\left(y_{2}\right) \Rightarrow F\left(x,y_{1}\right)  \geq F\left(x,y_{2}\right). $$ 
\end{definition}
Note that if $g$ is the identity mapping, then Definition \ref{def1} reduces to Definition 1.1 in \cite{Bha} of mixed monotone property.

\begin{definition} \label{def2} 
An element $ \left(x,y\right) \in X \times X $ is called a \textit{coupled coincidence point} of the mappings  $F:X \times X \rightarrow X$ and $g: X \rightarrow X$ if                                              
$$ 
F\left(x,y\right) = g\left(x\right) \textnormal{ and }  F\left(y,x\right) = g\left(y\right).
$$
\end{definition} 

\begin{definition} \label{def3} 
Let X be a non-empty  set. We say that the mappings $F:X \times X \rightarrow X$ and $g: X \rightarrow X$ commute if
$$
g(F(x,y))=F(g(x),g(y)),
$$
for all $x,y\in X$.
\end{definition}

Using basically these concepts, the results obtained in \cite{LakC} are some coincidence theorems and coupled common fixed point theorems obtained basically by considering a more general contractive condition than condition \eqref{Bhas} used in \cite{Bha}. The main result in \cite{LakC} is given by the next theorem.

\begin{theorem} \label{th2}
Let  $\left(X,\leq\right)$ be a partially ordered set and suppose there is a metric $d$  on $X$ such that $\left(X,d\right)$  is a complete metric space. Assume there exists a function  $\varphi:[0,\infty)\rightarrow [0,\infty)$ with $\varphi(t)< t$ and $\lim\limits_{r\rightarrow t}\psi(r)<t$ for all $t>0$ and also suppose  $F : X \times X \rightarrow X $ and $g : X \rightarrow X $ are such that $F$ has the  mixed $g$-monotone property and
\begin{equation} \label{Bhas1}
d\left(F\left(x,y\right),F\left(u,v\right)\right) \leq  \varphi\left(\frac{d\left(g(x),g(u)\right) + d\left(g(y),g(v)\right)}{2}\right),
\end{equation}
 for all $x,y,u,v\in X$ with $g(x) \geq g(u), g(y) \leq  g(v)$.  
Suppose $F(X\times X)\subset g(X)$, $g$ is continuous and commutes with $F$ and also suppose either

(a) $F$ is continuous or

(b) $X$ satisfy Assumption \ref{1.1}.

If there exist $x_{0}, y_{0} \in X$ such that  
\begin{equation} \label{mic3}
 g(x_{0}) \leq F\left(x_{0},y_{0}\right)\textrm{ and }g(y_{0}) \leq F\left(y_{0},x_{0}\right),
\end{equation}  then there exist $\overline{x}, \overline{y} \in X$ such that 
 $$
 g(\overline{x}) = F\left(\overline{x},\overline{y}\right)\textrm{and }g(\overline{y}) = F\left(\overline{y},\overline{x}\right),
 $$
 that is, $F$ and $g$ have a coupled coincidence.
	
\end{theorem}

Obviously, for $g=identity$ and $\varphi(t)=k t$, $0\leq k <1$, Theorem \ref{th2} reduces to Theorem \ref{th1}.

Starting from the results in \cite{LakC}  and \cite{Bha}, our main aim in this paper is to obtain more general coincidence point theorems and coupled common fixed point theorems for mixed monotone operators $F : X \times X \rightarrow X $  satisfying a contractive condition which is significantly more general that the corresponding conditions \eqref{Bhas1} and \eqref{Bhas} in \cite{LakC} and \cite{Bha}, respectively, thus extending many other related results in literature. 


\section{Main results}

Let $\Phi$ denote the set of all functions $\varphi:[0,\infty)\rightarrow [0,\infty)$ satisfying 

$(i_\varphi)$ $\varphi(t)<t$ for all $t\in (0,\infty)$;

$(ii_\varphi)$ $\lim\limits_{r\rightarrow t+}\varphi(r)<t$, for all $t\in (0,\infty)$.

The first main result in this paper is the following coincidence point theorem which generalizes Theorem 2.1 in \cite{LakC} and Theorem 2.1 in \cite{Bha}.

\begin{theorem}\label{th3}
	Let  $\left(X,\leq\right)$ be a partially ordered set and suppose there is a metric $d$  on $X$ such that $\left(X,d\right)$  is a complete metric space. Let $F : X \times X \rightarrow X $ be a  mixed $g$-monotone mapping for which there exist $\varphi\in \Phi$  such that for all $x,y,u,v\in X$ with $g(x) \geq g(u), g(y) \leq  g(v)$,
$$
d\left(F\left(x,y\right),F\left(u,v\right)\right)+d\left(F\left(y,x\right),F\left(v,u\right)\right) \leq 
$$                                                                       
\begin{equation} \label{Bhas2}
 \leq  2 \varphi\left(\frac{d\left(g(x),g(u)\right) + d\left(g(y),g(v)\right)}{2}\right).
\end{equation}  
Suppose $F(X\times X)\subset g(X)$, $g$ is continuous and commutes with $F$ and also suppose either

(a) $F$ is continuous or

(b) $X$ satisfy Assumption \ref{1.1}.

If there exist $x_{0}, y_{0} \in X$ such that  
\begin{equation} \label{mic}
 g(x_{0}) \leq F\left(x_{0},y_{0}\right)\textrm{ and }g(y_{0}) \geq F\left(y_{0},x_{0}\right),
\end{equation} or
\begin{equation} \label{mare} 
g(x_{0}) \geq F\left(x_{0},y_{0}\right)\textrm{ and }g(y_{0}) \leq F\left(y_{0},x_{0}\right),
\end{equation}
 then there exist $\overline{x}, \overline{y} \in X$ such that 
 $$
 g(\overline{x}) = F\left(\overline{x},\overline{y}\right)\textrm{and }g(\overline{y}) = F\left(\overline{y},\overline{x}\right),
 $$
 that is, $F$ and $g$ have a coupled coincidence. 
	\end{theorem}

\begin{proof}
Consider the functional $d_2:X^2\times X^2 \rightarrow \mathbb{R}_{+}$ defined by
$$
d_2(Y,V)=\frac{1}{2}\left[d(x,u)+d(y,v)\right],\,\forall Y=(x,y),V=(u,v) \in X^2.
$$
It is a simple task to check that $d_2$ is a metric on $X^2$ and, moreover, that, if $(X,d)$ is complete, then $(X^2,d_2)$ is a complete metric space, too.
Now consider the operator $T:X^2\rightarrow X^2$ defined by
$$
T(Y)=\left(F(x,y),F(y,x)\right),\,\forall Y=(x,y) \in X^2.
$$
Clearly, for $Y=(x,y),\,V=(u,v)\in X^2$, in view of the definition of $d_2$, we have
$$
d_2(T(Y),T(V))=\frac{d\left(F\left(x,y\right),F\left(u,v\right)\right)+d\left(F\left(y,x\right),F\left(v,u\right)\right)}{2}
$$
and
$$
d_2(Y,V)=\frac{d\left(x,u\right) + d\left(y,v\right)}{2}.
$$
Thus, by the contractive condition \eqref{Bhas2} we obtain that $F$ satisfies the following $\varphi$-contractive condition: 
\begin{equation} \label{contr}
d_2(T(Y), T(V))\leq \varphi\left(d_2(Y,V)\right),\,\forall Y\geq V \in X^2.
\end{equation} 
Assume \eqref{mic} holds (the case \eqref{mare} is similar). Then, there exists $x_0,y_0\in X$ such that
$$
g(x_0)\leq F(x_0,y_0) \textnormal{ and } g(y_0)\geq F(y_0,x_0). 
$$
Denote $Z_0=(g(x_0),g(y_0))\in X^2$ and consider the Picard iteration associated to $T$ and to the initial approximation $Z_0$, that is, the sequence $\{Z_n\}\subset X^2$ defined by
\begin{equation} \label{eq-3}
Z_{n+1}=T (Z_n),\,n\geq 0,
\end{equation}
where $Z_n=(g(x_n),g(y_n))\in X^2,\,n\geq 0$.

Since $F$ is $g$-mixed monotone, we have
$$
Z_0=(g(x_0),g(y_0))\leq (F(x_0,y_0), F(y_0,x_0))=(g(x_1),g(y_1))=Z_1
$$
and, by induction,
$$
Z_n=(g(x_n),g(y_n))\leq (F(x_n,y_n), F(y_n,x_n))=(g(x_{n+1}),g(y_{n+1}))=Z_{n+1},
$$
which actually shows that
\begin{equation} \label{eq-4.8}
g(x_n) \leq g(x_{n+1}) \textnormal{ and } g(y_n)\geq g(y_{n+1}), \textnormal{ for all } n\geq 0.
\end{equation}
Note also, in particular, that the mapping $T$ is monotone and the sequence $\{Z_n\}_{n=0}^{\infty}$ is non-decreasing.
Take $Y=Z_n\geq Z_{n-1}=V$ in \eqref{contr} and obtain
\begin{equation} \label{eq-4.2}
d_2(T(Z_{n}),T(Z_{n-1})\leq \varphi\left(d_2 (Z_n, Z_{n-1})\right),\,n\geq 1.
\end{equation}
This shows that the sequence $\{\delta_n\}_{n=1}^{\infty}$ given by
$$
\delta_n=d_2 (Z_n, Z_{n-1})=\frac{d(g(x_{n+1}),g(x_n))+d(g(y_{n+1}),g(y_n))}{2},\,n\geq 1,
$$
satisfies
\begin{equation} \label{eq-4.7}
\delta_{n+1}\leq \varphi(\delta_n), \textnormal{ for all } n\geq 1.
\end{equation}
From \eqref{eq-4.7} and $(i)_{\varphi}$ it follows that the sequence $\{\delta_n\}_{n=1}^{\infty}$  is non-increasing.  Therefore, there exists some $\delta \geq 0$ such that
\begin{equation} \label{eq-4.1}
\lim_{n\rightarrow \infty} \delta_n=\lim_{n\rightarrow \infty}\frac{d(g(x_{n+1}),g(x_n))+d(g(y_{n+1}),g(y_n))}{2}=\delta.
\end{equation}
We shall prove that $\delta=0$. Assume, to the contrary, that $\delta>0$. Then by letting $n\rightarrow \infty$ in \eqref{eq-4.7} we have
$$
\delta=\lim_{n\rightarrow \infty} \varphi(\delta_{n+1})\leq \lim_{n\rightarrow \infty} \varphi(\delta_{n})=\lim_{\delta_n\rightarrow \delta} \varphi(\delta_{n})<\delta,
$$
a contradiction. Thus $\delta=0$ and hence
\begin{equation} \label{eq-4}
\lim_{n\rightarrow \infty} \delta_n=\lim_{n\rightarrow \infty}\frac{d(g(x_{n+1}),g(x_n))+d(g(y_{n+1}),g(y_n))}{2}=0.
\end{equation}

We now prove that $\{Z_n\}_{n=0}^{\infty}$ is a Cauchy sequence in  $(X^2,d_2)$, that is, $\{g(x_n)\}_{n=0}^{\infty}$ and $\{g(y_n)\}_{n=0}^{\infty}$ are Cauchy sequences in $(X,d)$. Suppose, to the contrary, that at least one of the sequences $\{g(x_n)\}_{n=0}^{\infty}$, $\{g(y_n)\}_{n=0}^{\infty}$ is not a Cauchy sequence. Then there exists an $\epsilon>0$ for which we can find subsequences $\{g(x_{n(k)})\}$, $\{g(x_{m(k)})\}$ of $\{g(x_n)\}_{n=0}^{\infty}$ and $\{g(y_{n(k)})\}$, $\{g(y_{m(k)})\}$ of $\{g(y_n)\}_{n=0}^{\infty}$, respectively, with $n(k)>m(k)\geq k$ such that
\begin{equation} \label{eq-6}
\frac{1}{2}\left[d(g(x_{n(k)}),g(x_{m(k)}))+d(g(y_{n(k)}),g(y_{m(k)}))\right]\geq \epsilon, \,k=1,2,\dots.
\end{equation}
Note that we can choose $n(k)$ to be the smallest integer with property $n(k)>m(k)\geq k$ and satisfying \eqref{eq-6}. Then
\begin{equation} \label{eq-5.1}
d(g(x_{n(k)-1}),g(x_{m(k)}))+d(g(y_{n(k)-1}),g(y_{m(k)}))< \epsilon.
\end{equation}
By \eqref{eq-6} and \eqref{eq-5.1} and the triangle inequality we have
$$
\epsilon\leq r_k:=\frac{1}{2}\left[d(g(x_{n(k)}),g(x_{m(k)}))+d(g(y_{n(k)}),g(y_{m(k)}))\right]\leq
$$
$$
+\frac{d(g(x_{n(k)}),g(x_{n(k)-1}))+d(g(y_{n(k)}),g(y_{n(k)-1}))}{2}+
$$
$$ 
+\frac{d(g(x_{n(k)-1}),g(x_{m(k)}))+d(g(y_{n(k)-1}),g(y_{m(k)}))}{2}\leq
$$
$$ 
\leq \frac{d(g(x_{n(k)}),g(x_{n(k)-1}))+d(g(y_{n(k))},g(y_{n(k)-1}))}{2}+\epsilon.
$$
Letting $k\rightarrow \infty$ in the above inequality and using \eqref{eq-4} we get
\begin{equation} \label{eq-7.1}
\lim_{k\rightarrow \infty} r_k:=\lim_{k\rightarrow \infty} \frac{1}{2}\left[d(x_{n(k)},x_{m(k)})+d(y_{n(k)},y_{m(k)})\right]=\epsilon.
\end{equation}
On the other hand
$$
r_k:=\frac{d(g(x_{n(k)}),g(x_{m(k)}))+d(g(y_{n(k)}),g(y_{m(k)}))}{2}\leq
$$
$$
\leq \frac{d(g(x_{n(k)}),g(x_{n(k)+1}))+d(g(x_{n(k)+1}),g(x_{m(k)}))}{2}+
$$
$$
+\frac{d(g(y_{n(k)}),g(y_{n(k)+1}))+d(g(y_{n(k)+1}),g(y_{m(k)}))}{2}=
$$
$$
=\delta_{n(k)}+\frac{d(g(x_{n(k)+1}),g(x_{m(k)}))+d(g(y_{n(k)+1}),g(y_{m(k)}))}{2}\leq
$$
\begin{equation} \label{eq-7.2}
=\delta_{n(k)}+\delta_{m(k)}+\frac{d(g(x_{n(k)+1}),g(x_{m(k)+1}))+d(g(y_{n(k)+1}),g(y_{m(k)+1}))}{2}.
\end{equation}
Since $n(k)>m(k)$, by \eqref{eq-4.8} we have $g(x_{n(k)})\geq g(x_{m(k)})$ and $g(y_{n(k)})\leq g(y_{m(k)})$ and hence by \eqref{eq-4.2} one obtains
$$
d(g(x_{n(k)+1}),g(x_{m(k)+1}))+d(g(y_{n(k)+1}),g(y_{m(k)}+1))=
$$
$$
 =d\left(F(g(x_{n(k)}),g(y_{m(k)})\right),F\left(g(y_{n(k)}),g(y_{m(k)})\right)\leq 
$$
$$
\leq 2 \varphi\left(\frac{d(g(x_{n(k)}),g(x_{m(k)}))+d(g(y_{n(k)}),g(y_{m(k)}))}{2}\right)\leq  2 \varphi\left(r_{k}\right),
$$
which, by \eqref{eq-7.2}, yields
$$
r_k\leq \delta_{n(k)}+\delta_{m(k)}+\varphi(r_k).
$$
Letting $k\rightarrow \infty$ in the above inequality and using \eqref{eq-7.1} we get
$$
\epsilon \leq \lim_{k\rightarrow \infty} \varphi\left(r_{k}\right)=\lim_{r_k\rightarrow \epsilon+} \varphi\left(r_{k}\right)<\epsilon,
$$
a contradiction. This shows that $\{g(x_n)\}_{n=0}^{\infty}$ and $\{g(y_n)\}_{n=0}^{\infty}$ are indeed Cauchy sequences in the complete metric space $(X,d)$.

This implies there exist $\overline{x},\overline{y}$ in $X$ such that
\begin{equation} \label{eq-4.9}
\overline{x} =\lim_{n\rightarrow \infty} g(x_n),\, \overline{y} =\lim_{n\rightarrow \infty} g(y_n),
\end{equation} 
By \eqref{eq-4.9} and continuity of $g$,
\begin{equation} \label{eq-4.10}
\lim_{n\rightarrow \infty} g(g(x_n))=g(\overline{x}) \textnormal{ and } \lim_{n\rightarrow \infty} g(g(y_n)=g(\overline{y}).
\end{equation}
On the other hand, by \eqref{eq-3} and commutativity of $F$ and $g$,
\begin{equation} \label{eq-4.11}
g(g(x_{n+1}))=g(F(x_n,y_n))=F(g(x_n),g(y_n)),
\end{equation}
\begin{equation} \label{eq-4.12}
g(g(y_{n+1}))=g(F(y_n,x_n))=F(g(y_n),g(x_n)).
\end{equation}
We now prove that $g(\overline{x}) = F\left(\overline{x},\overline{y}\right)$ and $g(\overline{y}) = F\left(\overline{y},\overline{x}\right).$ 

Suppose first that assumption (a) holds. By letting $n\rightarrow \infty$ in \eqref{eq-4.11} and \eqref{eq-4.12}, in view of in \eqref{eq-4.9} and \eqref{eq-4.10}, we get 
$$g(\overline{x}) = \lim_{n\rightarrow \infty} g(g(x_{n+1}))=\lim_{n\rightarrow \infty} F(g(x_n),g(y_n))=F\left(\overline{x},\overline{y}\right)
$$ 
and, similarly
$$
g(\overline{y}) = \lim_{n\rightarrow \infty} g(g(y_{n+1}))=\lim_{n\rightarrow \infty} F(g(y_n),g(x_n))=F\left(\overline{y},\overline{x}\right),
$$
that is,  $(\overline{x},\overline{y})$ is a coincidence point of $F$ and $g$.

Suppose now assumption (b) holds. Since  $\{g(x_n)\}_{n=0}^{\infty}$ is a non-decreasing sequence that converges to $\overline{x}$, we have that $g(x_n)\leq \overline{x}$ for all $n$. Similarly, we obtain $g(y_n)\geq \overline{y}$ for all $n$.

Then, by triangle inequality and contractive condition \eqref{Bhas2},
$$
d(g(\overline{x}),F(\overline{x},\overline{y}))+d(g(\overline{y}),F(\overline{y},\overline{x}))\leq d(g(\overline{x}),g(g(x_{n+1})))+
$$
$$
+d(g(g(x_{n+1})),F(\overline{x},\overline{y}))+d(g(\overline{y}),g(g(y_{n+1})))+d(g(g(y_{n+1})),F(\overline{y},\overline{x}))=
$$
$$
=d(g(\overline{x}),g(g(x_{n+1})))+d(g(\overline{y}),g(g(y_{n+1})))+d(F(x_n,y_n),F(g(\overline{x}),g(\overline{y})))+
$$
$$
+d(F(y_n,x_n),F(g(\overline{y}),g(\overline{x})))\leq d(g(\overline{x}),g(g(x_{n+1})))+d(g(\overline{y}),g(g(y_{n+1})))+
$$
$$
+2 \varphi\left(\frac{d(g(x_n),g(\overline{x}))+d(g(y_n),g(\overline{y}))}{2}\right).
$$
Letting now $n\rightarrow \infty$ in the above inequality and taking into account that, by property $(i_{\varphi})$, $\lim\limits_{r\rightarrow 0+}\varphi(r)=0$, we obtain
$$
d(\overline{x},F(\overline{x},\overline{y}))+d(\overline{y},F(\overline{y},\overline{x}))=0
$$
which implies  that $d(g(\overline{x}),F(\overline{x},\overline{y}))=0$ and $d(g(\overline{y}),F(\overline{y},\overline{x}))=0$. 

\end{proof}

 \begin{remark} \em
Theorem \ref{th3} is more general than Theorem \ref{th2}, since the contractive condition \eqref{Bhas2} is weaker than \eqref{Bhas1}, a fact which is clearly illustrated by the next example.
 \end{remark}
\begin{example} \label{ex2} \em
Let $X=\mathbb{R}$ with $d\left(x,y\right)=|x-y|$ and natural ordering and let $g:X \rightarrow X$, $F:X\times X \rightarrow X$ be given by $g(x)=\frac{5x}{6},\,x\in X$ and
$$
F\left(x,y\right)=\frac{x-2y}{4}, \,(x,y)\in X^2.
$$
Then $F$ is $g$-mixed monotone, $F$ and $g$ commute and satisfy condition \eqref{Bhas2}  but  $F$ and $g$  do not satisfy condition \eqref{Bhas1}. Indeed, assume, to the contrary, that there exists $\varphi \in \Phi$, such that  \eqref{Bhas1} holds. This means
$$
\left|\frac{x-2y}{4}-\frac{u-2v}{4}\right| \leq \varphi\left(\frac{5}{6}\cdot \frac{\left|x-u\right|+\left|y-v\right|}{2}\right),\,x\geq u,\,y\leq v,
$$
by which, for $x=u$, $y< v$ and in view of $(i_{\varphi})$ we get
$$
\frac{1}{2}\left|y-v\right| \leq \varphi\left(\frac{5}{12}\left|y-v\right|\right)<\frac{5}{12}\left|y-v\right|<\frac{1}{2}\left|y-v\right|,
$$
a contradiction. Hence $F$ and $g$  do not satisfy condition \eqref{Bhas1}.

Now we prove that \eqref{Bhas2} holds. Indeed, we have
$$
\left|\frac{x-2y}{4}-\frac{u-2v}{4}\right| \leq \frac{1}{4}\left|x-u\right|+\frac{1}{2}\left|y-v\right|,\,x\geq u,\,y\leq v,
$$
and
$$
\left|\frac{y-2x}{5}-\frac{v-2u}{4}\right| \leq \frac{1}{4}\left|y-v\right|+\frac{1}{2}\left|x-u\right|,\,x\geq u,\,y\leq v,
$$
and by summing up the two inequalities above we get exactly  \eqref{Bhas2} with $\varphi(t)=\frac{3}{4}t$. Note also that $x_0=-3,\,y_0=3$ satisfy \eqref{mic}.

So  by Theorem \ref{th3} we obtain that $F$ has a (unique) coupled fixed point $(0,0)$, but Theorem \ref{th2} cannot be applied as $F$ and $g$  do not satisfy condition \eqref{Bhas1}.

\end{example}

The following corollary generalizes Theorem 2.1 in \cite{Bha} from coupled fixed points to coincidence points.
\begin{corollary} \label{cor1}
Let  $\left(X,\leq\right)$ be a partially ordered set and suppose there is a metric $d$  on $X$ such that $\left(X,d\right)$  is a complete metric space. Let $F : X \times X \rightarrow X $ be a mixed $g$-monotone mapping for which there exist $k\in [0,1)$ such that for all $x,y,u,v\in X$ with $g(x) \geq g(u), g(y) \leq  g(v)$,
$$
d\left(F\left(x,y\right),F\left(u,v\right)\right)+
$$
\begin{equation} \label{Bhas3}
+d\left(F\left(y,x\right),F\left(v,u\right)\right)  \leq k \left[d\left(g(x),g(u)\right) + d\left(g(y),g(v)\right)\right].
\end{equation}  
Suppose $F(X\times X)\subset g(X)$, $g$ is continuous and commutes with $F$ and also suppose either

(a) $F$ is continuous or

(b) $X$ satisfy Assumption \ref{1.1}.

If there exist $x_{0}, y_{0} \in X$ such that  
\begin{equation} \label{mic1}
 g(x_{0}) \leq F\left(x_{0},y_{0}\right)\textrm{ and }g(y_{0}) \leq F\left(y_{0},x_{0}\right),
\end{equation}  then there exist $\overline{x}, \overline{y} \in X$ such that 
 $$
 g(\overline{x}) = F\left(\overline{x},\overline{y}\right)\textrm{and }g(\overline{y}) = F\left(\overline{y},\overline{x}\right),
 $$
 that is, $F$ and $g$ have a coupled coincidence.
\end{corollary}

\begin{proof}
Take $\varphi(t)=k t$, $0\leq k<1$ in Theorem \ref{th3}.
\end{proof}

\begin{remark} \em
Let us note that, as suggested by Example \ref{ex2}, since the contractive condition \eqref{Bhas2} is valid  only for comparable elements in  $X^2$, Theorem \ref{th3} cannot guarantee in general the uniqueness of the coincidence point. 
\end{remark}

It is now our interest to identify additional conditions, like the ones  used in Theorem 2.2 of Bhaskar and Lakshmikantham \cite{Bha} or in Theorem 2.2 of Lakshmikantham and Ciric \cite{LakC}, to ensure that the coincidence fixed point guaranteed by Theorem \ref{th3} is unique. Such a condition is the one involved in the next theorem.

\begin{theorem} \label{th4}
In addition to  the  hypotheses  of Theorem \ref{th3}, suppose that for every $(\overline{x}, \overline{y}), (y^* , x^* ) \in X \times X$ there exists $(u,v)\in X\times X$ such that $(F(u,v),F(v,u))$ is comparable to $(F(x^*,y^*), F(y^*,x^*))$ and to $(F(\overline{x},\overline{y}), F(\overline{x},\overline{y}))$. Then $F$ and $g$ have a
unique coupled common fixed  point, that is, there exists a unique $(\overline{z},\overline{w})\in X^2$ such that
$$
\overline{z}=g(\overline{z})=F\left(\overline{z},\overline{w}\right)\textrm{and } \overline{w}=g(\overline{w})=F\left(\overline{w},\overline{z}\right).
$$
\end{theorem}

\begin{proof} From Theorem \ref{th3}, the set of coupled coincidences of $F$ and $g$ is nonempty. Assume that $Z^*=(x^*,y^*)\in X^2$ and $\overline{Z}=(\overline{x},\overline{y})$ are two coincidence points of $F$ and $g$. We shall prove that $g(x^*)=g(\overline{x})$ and $g(y^*)=g(\overline{y})$.

By hypothesis, there exists $(u,v)\in X^2$ such that $(F(u,v),F(v,u))$ is comparable to $(F(x^*,y^*), F(y^*,x^*))$ and to $(F(\overline{x},\overline{y}), F(\overline{x},\overline{y}))$. 
Put $u_0=u,\,v_0=v$ and choose $u_1,v_1\in X$ so that $g(u_{1})=F(u_0,v_0),\,g(v_{1})=F(v_0,u_0)$. Then, similarly to the proof of Theorem \ref{th3}, we obtain the sequences $\{g(u_n)\}$, $\{g(v_n)\}$ defined as follows:
$$
g(u_{n+1})=F(u_n,v_n),\,g(v_{n+1})=F(v_n,u_n),\,n\geq 0.
$$
Now construct in the same manner the sequences $\{g(x_n)\}$, $\{g(y_n)\}$, $\{g(x^*_n)\}$, $\{g(y^*_n)\}$, by setting $x_0=\overline{x}$, $y_0=\overline{y}$, $x^*_0=\overline{x}$ and $y^*_0=\overline{y}$, respectively. This means that, for all $\,n\geq 0$,
$$
g(x_{n})=F(\overline{x},\overline{y}),\,g(y_{n})=F(\overline{y},\overline{x});\,g(x^*_{n})=F(x^*_0,y^*_0),\,g(y^*_{n})=F(y^*_0,x^*_0).
$$
Since  $(F(\overline{x},\overline{y}),F(\overline{y},\overline{x}))=(g(x_1),g(y_1))=(g(\overline{x}),g(\overline{y}))$ and \\ $(F(u,v),F(v,v))=(g(u_1),g(v_1))$ are comparable, it follows that $g(\overline{x})\leq g(u_1)$ and $g(\overline{y})\geq g(v_1)$. 

Further, we easily show that $g((\overline{x}),g(\overline{y}))$ and $(g(u_n),g(v_n))$ are comparable, that is, $g(\overline{x})\leq g(u_n)$ and  $g(\overline{y})\geq g(v_n)$, for all $n\geq 1$.

Thus, by the contractive condition \eqref{Bhas2}, one get
$$
\frac{d(g(\overline{x}),g(u_{n+1}))+d(g(\overline{y}),g(v_{n+1}))}{2}=
$$
$$
=\frac{d(F(\overline{x},\overline{y}), F(u_n,v_n))+d(F(\overline{y}, \overline{x}),F(v_n,u_n))}{2}\leq
$$
\begin{equation} \label{eq-11}
\leq \varphi\left(\frac{d(g(\overline{x}),g(u_{n}))+d(g(\overline{y}),g(v_{n}))}{2}\right).
\end{equation}
Thus, by \eqref{eq-11}, we deduce that the sequence $\{\Delta_n\}$ defined by
$$
\Delta_n=\frac{d(g(\overline{x}),g(u_{n}))+d(g(\overline{y}),g(v_{n}))}{2},\,n\geq 0,
$$
satisfies
\begin{equation} \label{eq-12}
\Delta_{n+1}\leq \varphi(\Delta_n),\,n\geq 0.
\end{equation}
Now use $(i_{\varphi})$ to deduce by \eqref{eq-12} that  $\{\Delta_n\}$ is non-decreasing, hence convergent to some $\delta\geq 0$.

We shall prove that $\delta= 0$. Assume, to the contrary, that $\delta> 0$. Then, we can find a $n_0$ such that $\Delta_n\geq \delta >0$, for all $n\geq n_0$. By letting $n\rightarrow \infty$ in \eqref{eq-12} we obtain, in view of $(ii_{varphi})$, that
$$
\delta \leq \lim_{n\rightarrow \infty} \varphi(\Delta_n)=\lim_{\Delta_n\rightarrow \delta+} \varphi(\Delta_n)<\delta,
$$
a contradiction. Therefore
  $d(g(\overline{x}),g(u_{n+1}))+d(g(\overline{y}),g(v_{n+1}))\rightarrow0$ as $n\rightarrow \infty$, that is,
\begin{equation} \label{eq-13}
\lim_{n\rightarrow \infty} d(g(\overline{x}),g(u_{n+1}))=0,\textnormal{ and } \lim_{n\rightarrow \infty} d(g(\overline{y}),g(v_{n+1}))=0.
\end{equation}
Similarly, we obtain that
\begin{equation} \label{eq-14}
\lim_{n\rightarrow \infty} d(g(x^*),g(u_{n+1}))=0,\textnormal{ and } \lim_{n\rightarrow \infty} d(g(y^*),g(v_{n+1}))=0.
\end{equation}
By \eqref{eq-13} and \eqref{eq-14} and the triangle inequality, we have
$$
d(g(\overline{x}),g(x^*))\leq d(g(\overline{x}),g(u_{n+1}))+d(g(x^*),g(u_{n+1}))\rightarrow 0 \textnormal{ as } n\rightarrow \infty,
$$
$$
d(g(\overline{y}),g(y^*))\leq d(g(\overline{y}),g(v_{n+1}))+d(g(y^*),g(v_{n+1}))\rightarrow 0 \textnormal{ as } n\rightarrow \infty.
$$
Hence 
\begin{equation} \label{eq-15}
g(\overline{x})=g(x^*) \textnormal{ and } g(\overline{y})=g(y^*),
\end{equation}
that is, $F$ and $g$ have a unique coupled coincidence. Now we shall prove that actually $F$ and $g$ have a unique coupled common fixed point. Since
$$
g(\overline{x})=  F\left(\overline{x},\overline{y}\right)\textnormal{ and } g(\overline{y})= F\left(\overline{y},\overline{x}\right),
$$
and $F$ and $g$ commutes, we have
\begin{equation} \label{eq-16}
g(g(\overline{x}))=  g(F\left(\overline{x},\overline{y}\right))= F\left(g(\overline{x}),g(\overline{y})\right),
\end{equation}
and
\begin{equation} \label{eq-17}
g(g(\overline{y}))=  g(F\left(\overline{y},\overline{x}\right))= F\left(g(\overline{y}),g(\overline{x})\right).
\end{equation}
Denote $g(\overline{x})=\overline{z}$ and $g(\overline{y})=\overline{w}$. Then, by \eqref{eq-16} and \eqref{eq-17} one gets
$$
g(\overline{z})=F\left(\overline{z},\overline{w}\right) \textnormal{ and } g(\overline{w})=F\left(\overline{w},\overline{z}\right).
$$
Thus, $\left(\overline{z},\overline{w}\right)$ is a coupled coincidence point of $F$ and $g$. Then, by \eqref{eq-15} with $x^*=\overline{z}$ and $y^*=\overline{w}$, it follows that $g(\overline{z})=g(\overline{x})$ and $g(\overline{w})=g(\overline{y})$, that is
\begin{equation} \label{eq-18}
g(\overline{z})=\overline{z} \textnormal{ and } g(\overline{w})=\overline{w}. 
\end{equation}
Now from \eqref{eq-15} and \eqref{eq-18} we get
$$
\overline{z}=g(\overline{z})=F\left(\overline{z},\overline{w}\right) \textnormal{ and } \overline{w}=g(\overline{w})=F\left(\overline{w},\overline{z}\right).
$$
Therefore $\left(\overline{z},\overline{w}\right)$ is a coupled common fixed point of $F$ and $g$.

To prove the uniqueness, assume $(p,q)$ is another coupled common fixed point of $F$ and $g$. Then by \eqref{eq-18} we have
$$
p=g(p)=g(\overline{z})=\overline{z} \textnormal{ and } q=g(q)=g(\overline{w})=\overline{w}.
$$
\end{proof}

\begin{corollary} \label{cor3}
In addition to  the  hypotheses  of Corollary \ref{cor1},  suppose that for every $(\overline{x}, \overline{y}), (y^* , x^* ) \in X \times X$ there exists $(u,v)\in X\times X$ such that $(F(u,v),F(v,u))$ is comparable to $(F(x^*,y^*), F(y^*,x^*))$ and to $(F(\overline{x},\overline{y}), F(\overline{x},\overline{y}))$. Then $F$ and $g$ have a
unique coupled common fixed  point.
\end{corollary}

\begin{corollary} \label{cor2}
Let  $\left(X,\leq\right)$ be a partially ordered set and suppose there is a metric $d$  on $X$ such that $\left(X,d\right)$  is a complete metric space. Let $F : X \times X \rightarrow X $ be a mixed monotone mapping for which there exist $k\in [0,1)$ such that for all $x,y,u,v\in X$ with $x \geq u, y \leq  v$,
$$
d\left(F\left(x,y\right),F\left(u,v\right)\right)+
$$
\begin{equation} \label{Bhas4}
+d\left(F\left(y,x\right),F\left(v,u\right)\right)  \leq k \left[d\left(x,u\right) + d\left(y,v\right)\right].
\end{equation}  
Suppose  either

(a) $F$ is continuous or

(b) $X$ satisfy Assumption \ref{1.1}.

If there exist $x_{0}, y_{0} \in X$ such that  \eqref{mic} is satisfied, 
then $F$ has a coupled fixed point.
\end{corollary}

\begin{proof}
Take $g(x)=x$ and $\varphi(t)=k t$, $0\leq k<1$ in Theorem \ref{th3}.
\end{proof}

\begin{remark} \em

Corollary \ref{cor2} is a generalization of Theorem \ref{th1} (Theorem 2.1 in \cite{Bha}). Note also that in \cite{Bha} and \cite{LakC} the authors use only condition \eqref{mic}, although the alternative assumption \eqref{mare} is also applicable.
\end{remark}

\begin{remark}\em
As a final conclusion, we note that our results in this paper improve all coincidence point theorems and coupled fixed point theorems in \cite{LakC} and \cite{Bha}, and also many other related results: \cite{Ber11a}-\cite{Luong}, for coupled fixed point results and \cite{Aga}, \cite{Nie06}-\cite{Ran}, for fixed point results, by considering the more general (symmetric) contractive condition \eqref{Bhas2}.  By replacing the commutativity of $F$ and $g$ by the more general property "$F$ and $g$ are compatible", we can also extend the results in \cite{Choud}. This will be done in a forthcoming paper.
\end{remark}

\vskip 0.5 cm {\it 

Department of Mathematics and Computer Science

North University of Baia Mare

Victoriei 76, 430122 Baia Mare ROMANIA

E-mail: vberinde@ubm.ro}

\begin{thebibliography}{00}

\bibitem{Aga}  Agarwal, R.P., El-Gebeily, M.A. and O'Regan, D., \textit{Generalized contractions in partially ordered metric spaces}, Appl. Anal. \textbf{87} (2008) 1–-8

\bibitem{Ber07}  Berinde, V., \textit{Iterative approximation of fixed points}. Second edition, Lecture Notes in Mathematics, 1912, Springer, Berlin, 2007

\bibitem{Ber11a}  Berinde, V., \textit{Generalized coupled fixed point theorems for mixed monotone mappings in partially ordered metric spaces} (submitted)

\bibitem{Choud} Choudhury, B. S., Kundu, A., \textit{A coupled coincidence point result in partially ordered metric spaces for
compatible mappings}, Nonlinear Anal. \textbf{73} (2010) 2524–-2531
  

\bibitem{Bha}  Bhaskar, T. G., Lakshmikantham, V., \textit{Fixed point theorems in partially ordered metric spaces and applications}, Nonlinear Anal. \textbf{65} (2006), no. 7, 1379--1393 
 
\bibitem{LakC} Lakshmikantham, V., \' Ciri\' c, L.,  \textit{Coupled fixed point theorems for nonlinear contractions in partially ordered metric spaces}, Nonlinear Anal. \textbf{70} (2009), 4341–-4349

\bibitem{Luong} Luong N. V., Thuan, N. X., \textit{Coupled fixed points in partially ordered metric spaces and application}, Nonlinear Anal.,  \textbf{74} (2011), 983--992

\bibitem{Nie06} Nieto, J. J., Rodriguez-Lopez, R.,\ \textit{Contractive mapping theorems in partially ordered sets and applications to ordinary differential  equations}, Order \textbf{22} (2005), no. 3, 223-239 (2006) 
 
\bibitem{Nie07} Nieto, J. J., Rodriguez-Lopez, R.,\ \textit{Existence and uniqueness of fixed point in partially ordered sets  and applications to ordinary differential  equations}, Acta. Math. Sin., (Engl. Ser.) \textbf{23}(2007), no. 12, 2205--2212
 
\bibitem{Ran} Ran, A. C. M., Reurings, M. C. B., \textit{A fixed point theorem in partially ordered sets and some applications to matrix equations}, Proc. Amer. Math. Soc. \textbf{132} (2004), no. 5, 1435--1443
  
\bibitem {Rus2}  Rus, I. A., Petru\c sel, A., Petru\c sel, G., \textit{Fixed Point Theory}, Cluj University Press, Cluj-Napoca, 2008

\end{thebibliography}
\end{document}